\documentclass[a4paper,11pt]{amsart}

\usepackage{amsmath}
\usepackage{mathtools}
\usepackage[all]{xy}
\usepackage[latin1]{inputenc}
\usepackage{varioref}
\usepackage{amsfonts}
\usepackage{amssymb}
\usepackage{a4wide}
\usepackage{bbm}
\usepackage{graphicx}
\usepackage{mathrsfs}
\usepackage[hypertexnames=false,backref=page,pdftex,
 	pdfpagemode=UseNone,
 	breaklinks=true,
 	extension=pdf,
 	colorlinks=true,
 	linkcolor=blue,
 	citecolor=red,
 	urlcolor=blue,
 ]{hyperref}

\newcommand{\sB}{{\mathcal B}}
\newcommand{\sC}{{\mathcal C}}

\newcommand{\sH}{{\mathcal H}}

\newcommand{\sJ}{{\mathcal J}}

\newcommand{\sM}{{\mathcal M}}

\newcommand{\sO}{{\mathcal O}}

\newcommand{\sQ}{{\mathcal Q}}
\newcommand{\sR}{{\mathcal R}}
\newcommand{\sS}{{\mathcal S}}

\newcommand{\scrR}{{\mathscr R}}

\newcommand{\C}{{\mathbb C}}

\renewcommand{\P}{{\mathbb P}}
\newcommand{\Q}{{\mathbb Q}}

\newcommand{\Z}{{\mathbb Z}}

\newcommand{\abs}[1]{{\left|#1\right|}}

\newcommand{\Hilb}{{\rm Hilb}}

\newcommand{\img}{\operatorname{im}}
\newcommand{\into}{{\, \hookrightarrow\,}}
\newcommand{\isom}{{\ \cong\ }}

\newcommand{\ohne}{{\ \setminus \ }}

\newcommand{\ratl}{\dashrightarrow}

\newcommand{\reg}{{\operatorname{reg}}}

\newcommand{\rk}{{\rm rk}}
\newcommand{\sing}{{\operatorname{sing}}}
\newcommand{\Sing}{\operatorname{Sing}}

\renewcommand{\to}[1][]{\xrightarrow{\ #1\ }}

\newcommand{\tensor}{\otimes}

\newcommand{\vphi}{\varphi}

\newcommand{\wt}[1]{{\widetilde{#1}}}

\newtheoremstyle{citing}
  {}
  {}
  {\itshape}
  {}
  {\bfseries}
  {\textbf{.}}
  {.5em}
  {\thmnote{#3}}

\theoremstyle{plain}

\newtheorem{theorem}[subsection]{Theorem}

\theoremstyle{definition}

\newtheorem{corollary}[subsection]{Corollary}

\newtheorem{example}[subsection]{Example}
\newtheorem{lemma}[subsection]{Lemma}
\newtheorem{proposition}[subsection]{Proposition}
\numberwithin{equation}{section}

\theoremstyle{remark}
\newtheorem{remark}[subsection]{Remark}

{\theoremstyle{citing}
\newtheorem*{custom}{}}

\newcommand{\mult}{{\operatorname{mult}}}
\newcommand{\univ}{{\operatorname{univ}}}
\newcommand{\erz}{{\operatorname{span}\ }}
\newcommand{\Gr}{{\operatorname{Gr}}}
\newcommand{\Flag}{{\operatorname{Flag}}}

\title[Singular cubics]{Twisted cubics on singular cubic fourfolds~--- \\ On Starr's fibration}

\author{Christian Lehn}
\address{Christian Lehn\\Institut f\"ur Algebraische Geometrie\\
Gottfried Wilhelm Leibniz Universit\"at Hannover\\Welfengarten 1\\30167 Hannover\\Germany}
\email{lehn@math.uni-hannover.de}

\let\origmaketitle\maketitle
\def\maketitle{
  \begingroup
  \let\MakeUppercase\relax 
  \origmaketitle
  \endgroup
}

\begin{document}
\thispagestyle{empty}

\maketitle

\begin{abstract}
We study lines and twisted cubics on cubic fourfolds with simple isolated singularities. We show that the Hilbert scheme compactification of the total space of Starr's fibration on the space of twisted cubics on a cubic fourfold with simple isolated singularities and not containing a plane admits a contraction to a singular projective symplectic variety of dimension eight. The latter admits a crepant resolution which is a deformation of the symplectic eightfold constructed from the space of twisted cubics on a smooth cubic fourfold. Thereby we give another proof of the fact that the latter is a deformation of the Hilbert scheme of four points on a K3 surface. We also show that the variety of lines on a cubic fourfold with simple singularities is a singular symplectic variety birational to the Hilbert scheme of two points on a K3 surface.
\end{abstract}

\setlength{\parindent}{0em}
\setcounter{tocdepth}{1}

\tableofcontents

\section{Introduction}\label{sec intro}
\thispagestyle{empty}

By \cite{LLSvS} the space $M_3(Y)$ of generalized twisted cubics on a smooth cubic fourfold $Y \subset \P^5$ not containing a plane admits a regular contraction to an irreducible holomorphic symplectic manifold $Z(Y)$. The latter turned out to be deformation equivalent to the Hilbert scheme of four points on a K3 surface by \cite{AL}.

In this note we study $M_3(Y)$ for cubic fourfolds $Y$ with simple isolated singularities and not containing a plane. In this case we can also produce a variety $Z(Y)$ 
and it will be a singular symplectic variety in the sense of Beauville \cite{B}. 

\begin{custom}[Theorem \ref{theorem main}]
If $Y$ is a cubic fourfold with simple isolated singularities and if $Y$ does not contain a plane, then there are morphisms
\[
M_3(Y) \to[a] Z'(Y) \to[b] Z(Y)
\]
such that $a$ is a smooth $\P^2$-fibration and $b$ is birational. Here $Z(Y)$ is a symplectic variety of dimension $8$. Moreover, there is a closed immersion $\iota: Y \into Z(Y)$ such that $Y$ is not contained in the singular locus of $Z(Y)$ and $Y$ is a Lagrangian subvariety. 
\end{custom}

By work of Namikawa it follows that $Z$ has a symplectic resolution which is a deformation of the variety constructed in \cite{LLSvS}. 

\begin{custom}[Corollary \ref{corollary namikawa}]
If $Y$ is a cubic fourfold with simple isolated singularities and if $Y$ does not contain a plane, then the variety $Z(Y)$ admits a symplectic resolution, which is deformation equivalent to $Z(Y')$ for a smooth cubic fourfold $Y'$ not containing a plane.
\end{custom}

The spaces $M_3$ and $Z$ are compactifications of Starr's rational fibrations $M_3^\circ \ratl Z^\circ \ratl \P^4$. As Starr has shown, $Z^\circ \ratl \P^4$ is birational to a Jacobian fibration, see section \ref{sec faserung}, which is known to be an irreducible symplectic manifold deformation equivalent to a Hilbert scheme of points on a K3 surface. 
Thus using Huybrechts' theorem we obtain another proof that the irreducible symplectic manifolds $Z(Y)$ for smooth $Y$ are deformations of a Hilbert scheme of points on a K3 surface.

\begin{custom}[Corollary \ref{corollary starr}]
If $Y$ is a smooth cubic fourfold not containing a plane, then $Z_3(Y)$ is deformation equivalent to the Hilbert scheme of $4$ points on a K3 surface.
\end{custom}

On the way, we also study the Fano varieties of lines $M_1(Y)$ for singular cubic fourfolds and obtain singular symplectic varieties. A cubic fourfold with simple singularities has an associated K3 surface, which is the minimal resolution of the variety of lines through a singular point.

\begin{custom}[Theorem \ref{theorem m1 birational to hilb}]
Let $Y \subset \P^5$ be a cubic hypersurface with only simple isolated singularities and let $S$ be an associated K3 surface. Then the Fano variety $M_1(Y) \subset \Gr(V,2)$ is a singular symplectic variety, which is birational to $\Hilb^2(S)$.
\end{custom}

Certainly, we may draw the same conclusion as in Corollary \ref{corollary namikawa}, namely that $M_1(Y)$ has a symplectic resolution which is a deformation of the variety $M_1(Y')$ of lines on a smooth cubic fourfold $Y'$, see Corollary \ref{corollary namikawa for m1}. The variety $M_1(Y')$ was shown to be an irreducible symplectic manifold in \cite{BD}.

\subsection*{Structure of the article}
The line of argument is the same for $M_1$ and $M_3$ and works as follows. Connectedness is transferred from the smooth case via a Stein factorization argument and normality (in particular, irreducibility) is shown by a dimension estimate of the singular locus. The de Jong-Starr construction gives a $2$-form on the smooth part of $M_1$ respectively the dense open subset of $M_3$ parametrizing aCM-curves. Section \ref{sec recap} recalls the construction in the smooth case, in particular we observe that smoothness of $Y$ is not needed to obtain the $\P^2$-fibration. The $2$-form on $M_3$ thus descends to a generically nondegenerate $2$-form on $Z'$. Section \ref{sec kubiken} is logically independent and treats $M_1$. First, we establish some basic facts about singular cubic fourfolds and the lines they contain. The new result in this section is the proof of normality as sketched above and of symplecticity in the sense of Beauville. For the latter it has to be shown that the symplectic form on the smooth part extends holomorphically over exceptional divisors on some resolution of singularities. This is shown using symplecticity of an explicit birational model of $M_1$, namely a Hilbert scheme of points on the associated K3 surface. Examination of $M_1$ is also necessary to obtain the dimension estimates on the singular locus of $M_3$ which is done in section \ref{sec divisoren}. The issue is to bound the dimension of the locus of $3$-planes whose intersection with $Y$ is a badly singular cubic surface. The argument sketched above is carried out for $M_3$ in section \ref{sec kontraktion}. As in the smooth case one faces the difficulty that the divisor of non-CM curves is a degeneracy divisor for the symplectic form, but it can be contracted. The extension property of the resulting $2$-form is verified as for $M_1$ using an explicit birational model, namely Starr's fibration, whose construction is explained in section~\ref{sec faserung}. Finally, we deduce that $M_3$ of a smooth cubic is a deformation of the Hilbert scheme of four points on a K3 surface using Namikawa's theory of deformations of singular symplectic varieties and Huybrecht's result on deformation equivalence of birational hyperk\"ahler manifolds.

\subsection*{Acknowledgements} 
It is a great pleasure to thank Jason Starr for explaining me his construction of the Lagrangian fibration, for his generosity in sharing insights, for helpful discussions and for pointing out several questions. Furthermore, I greatly benefited from discussions with Radu Laza and Claire Voisin.

The author gratefully acknowledges the support by the DFG; the major part of this article was written when the author was supported through the DFG research grant Le 3093/1-1. The article was completed when the author was supported through the DFG research grant Le 3093/1-2.

\subsection{Notations}

We denote by $M_d(\P^5)$ the Hilbert scheme compactification of the space of rational normal curves of degree $d$ in $\P^5$ and we call its elements \emph{generalized rational normal curves}. For a closed subvariety variety $X\subset \P^5$ we denote by $M_d(X)$ the space of generalized rational normal curves of degree $d$ on $X$, that is $M_d(X)=M_d(\P^5) \cap \Hilb(X)$. Thus, $M_1(X)$ is just the Fano variety of lines in $X$. We will write $M_d(X,p)$ for the subscheme of generalized rational normal curves passing through a given point $p\in X$.

If $V$ is a locally free sheaf on a quasi-projective $\C$-scheme $S$, we denote by $\sO_V(1)$ the universal rank one quotient of $V$ on $\P_S(V)$. For a subscheme $X \subset \P^n$ we will denote by $P_X:=\erz X$ the smallest linear subspace of $\P^n$ containing $X$. 

\section{The \texorpdfstring{$\P^2$}--fibration}\label{sec recap}
In this section we review the construction from \cite{LLSvS}. Let $Y\subset \P^5=\P(V)$ be a smooth cubic fourfold not containing a plane and denote by $M_3(Y) \subset \Hilb^{3n+1}(Y)$ be the component of the Hilbert scheme of $Y$ containing twisted cubics. There are two morphisms 
\[
M_3(Y) \to[a] Z'(Y) \to[b] Z(Y),
\]
where $a$ is a smooth $\P^2$ fibration and $b$ is birational, see \cite[Theorem B]{LLSvS}. We will suppress the dependence on $Y$ if there is no danger of confusion. 

Every curve $C\in M_3$ spans a $\P^3$ inside $\P^5$ and so we obtain a morphism
\[
s:M_3 \to G:=\Gr(V,4),\quad C\mapsto P_C = \erz C,
\]
see \cite[\S 1]{LLSvS}. The fibration $a$ is constructed \emph{universally} over the Grassmanian $G$. This means the following. Fix a four dimensional quotient $V\to W$ and denote by $S\subset \P(S^3W^\vee)$ the subset of the space of cubic surfaces corresponding to integral cubic surfaces. Let $H_0$ be the component of the Hilbert scheme of $\P^3$ containing twisted cubics and let $R \subset S \times H_0$ be the incidence variety of curves on (integral) cubic surfaces. Let $H$ be the variety from \cite[Thm 3.13]{LLSvS} such that we have morphisms $R\to[a_\univ] H \to[\delta_\univ] S$ where $a_\univ$ is a smooth $\P^2$-fibration and $\delta_\univ$ is generically finite. We let the space $\P(W)$ vary in $\P(V)$ and obtain the same picture in the relative situation over the Grassmanian : 
\begin{equation}
\label{eq universal p2fibration}
\sR\to[a_\univ] \sH \to[\delta_\univ] \sS \to G.
\end{equation}
A cubic fourfold $Y$ not containing a plane gives rise to a section $\gamma_Y$ of the bundle projection $\sS\to G$ via $W\mapsto Y \cap \P(W)$ and it can easily be verified that $M_3(Y)= \sR \times_{\sS} G$.
Note that the smoothness hypothesis on $Y$ is not needed in \cite[Proposition 4.1]{LLSvS}, thus we obtain

\begin{proposition}\label{proposition p2fibration}
If $Y$ is a cubic fourfold, not necessarily smooth, but not containing a plane, then there is a diagram 
\[
\xymatrix{
M_3(Y) \ar[r]\ar[d]_a & \sR \ar[d]^{a_\univ}\\
Z'(Y) \ar[r] \ar[d]_\delta &\sH \ar[d]^{\delta_\univ}\\
G \ar[r]^{\gamma_Y} &\sS \\
}
\]
where all squares are cartesian. \qed 
\end{proposition}

Finally, we recall the dichotomy for curves in $M_3(\P^5)$: there is a Zariski open subset parametrizing curves that are \emph{arithmetically Cohen-Macaulay} (aCM). This property, which is stronger than being Cohen-Macaulay, means that the affine cone over such a curve is Cohen-Macaulay. The complement of the set of aCM curves is a smooth divisor $J(\P^5) \subset M_3(\P^5)$, whose points correspond to curves that are not Cohen-Macaulay (non-CM curves). We refer to \cite[\S 1]{LLSvS} and references therein. We record the following properties of non-CM curves on cubic hypersurfaces.

\begin{proposition}\label{proposition cartier divisor}
In the situation of Proposition \ref{proposition p2fibration} there is an effective Cartier divisor $D \subset Z'(Y)$ such that its preimage in $M_3(Y)$ is exactly the Cartier divisor $J(Y) \subset M_3(Y)$ of all non-CM curves.
\end{proposition}

Let $C \subset Y$ be a non-CM curve. Then it is given inside $P_C=\P^3$ by a singular plane cubic $C' \subset C$ together with an embedded point supported some $p\in (C')^\sing$, see \cite[\S 1, (2)]{LLSvS}. If the plane $P_{C'}$ is not contained in $Y$, then $C'=Y\cap P_{C'}$ so that $C'$ is determined by $P_{C'}$. Note also that the cubic surface $S_C=Y \cap P_C$ is necessarily singular at $p$ as the Zariski tangent space of $C$ at $p$ has dimension $3$. Consequently, $P_C$ is contained in the tangent space $T_pY \subset \P^5$. Conversely, given $p \in Y$ and linear subspaces $P_2 \subset P_3 \subset T_p Y$ of the indicated dimensions, we can produce a non-CM curve in $M_3(Y)$. This yields a bijective morphism
\[
\Flag(\Omega_Y,3,2) \to J(Y), \quad (p,P_2,P_3) \mapsto C
\]
where $\Flag(\Omega_Y,3,2)$ denotes the variety of partial  flags, that is, of successive locally free quotients of $\Omega_Y$ of the indicated dimensions. There is a canonical forgetful map to the Grassmanian $\Flag(\Omega_Y,3,2) \to T:=\Gr_Y(\Omega_Y,3)$.

\begin{proposition}\label{proposition non cm}
Let $Y \subset \P^5$ be a cubic fourfold with isolated singularities and not containing a plane. Then there is a commutative diagram
\[
\xymatrix{
\Flag(\Omega_Y,3,2) \ar[r] \ar[d] & J(Y)\ar[d] \\
T \ar[r] & D\\
}
\]
where the horizontal morphisms are isomorphisms. In particular, $J(Y) \subset M_3(Y)$ and $D \subset Z'(Y)$ are reduced.
\end{proposition}
\begin{proof}
The argument for the upper isomorphism is the same as in \cite[Lemma 1.1]{LLSvS}. Note that $\Flag(\Omega_Y,3,2)$ is reduced. The lower isomorphism is obtained by descent.
\end{proof}
\section{Singular cubics and lines on them}\label{sec kubiken}
We record some basics on singular cubics. Let $Y \subset \P^5$ be a cubic hypersurface. 
A frequently used tool will be the variety $M_1(Y,p)$ of lines through a given point $p\in Y$. Projection from $p$ gives a natural embedding of this variety into the $\P^4$ of lines through $p \in \P^5$ which we will identify with a hyperplane $H\subset \P^5$ not containing $p$. Hence, we have embeddings $M_1(Y,p) \subset X=Y \cap H \subset Y$.
\begin{proposition}\label{proposition lines through p}
Let $Y \subset \P^5$ be an integral cubic hypersurface and fix a point $p \in Y$. If $p$ is a singular point of $Y$, then one of the following holds
\begin{enumerate}
	\item $\dim M_1(Y,p)=3$ and $Y$ is a cone over a cubic threefold with vertex~$p$.
	\item $\dim M_1(Y,p)=2$, more precisely it  is the complete intersection of a cubic and a quadric in $\P^4$.
\end{enumerate}
If $p$ is a non-singular point of $Y$, then one of the following holds
\begin{enumerate}
\setcounter{enumi}{2}
	\item $\dim M_1(Y,p)=1$, more precisely it  is the complete intersection of a cubic and a quadric in $\P^3$
	\item\label{item plane} $\dim M_1(Y,p)=2$ and $Y$ contains a plane.
\end{enumerate} 
\end{proposition}
\begin{proof}
If $\dim M_1(Y,p)=3$, the subvariety of $Y$ swept out by these lines has dimension $4$ and as $Y$ is integral it has to coincide with $Y$. Thus, $Y$ is a cone with vertex $p$. Choose a coordinate system $x_0,\ldots, x_5$ on $\P^5$ such that $p$ has coordinates $[0:\ldots : 0 :1]$. We write the equation $f$ of $Y$ as
\[
f= f_1 x_5^2 + f_2 x_5 + f_3, 
\]
where $f_i \in \C[x_0,\ldots, x_4]$ is a homogeneous polynomial of degree $i$. Then $p \in Y$ is singular if and only if $f_1 = 0$. Let us denote by $H_i\subseteq \P^5$ the subscheme defined by $f_i$. The variety $M_1(Y,p)$ can be realized in $\P^4=\{x_5=0\}$ as the intersection $H_1\cap H_2\cap H_3$.
Thus, if $p$ is a singular point and $Y$ is not a cone, then $\dim M_1(Y,p)\leq 2$ and it is also the intersection $H_2\cap H_3$. It remains the case where $p$ is non-singular. As above $\dim M_1(Y,p)\leq 2$ and it is given by the intersection of $H_2$ and $H_3$ in $\P^3=\P^4\cap H_1$. If $\dim M_1(Y,p) =  2$, then it is given either by $H_2$ or by $H_3$ or $H_2$ and $H_3$ have a common component. In this case, $M_1(Y,p)$ contains a line so that the cone over this line is a plane contained in $Y$. This completes the proof.
\end{proof}
\begin{example}
The polynomial $f=x_5^2x_4 + x_5 x_4^2 + x_0^3 + x_1^3 + x_2^3 +  x_3^3$ defines a smooth cubic hypersurface $Y \subset \P^5$ and $p=(0,0,0,0,0,1)$ lies on $Y$. Then $Y$ contains a plane, for example given by $x_4 = x_0-x_1 = x_2-x_3=0$, and $M_1(Y,p)$ is a cubic surface and can be realized by $x_4=x_5=0=x_0^3 + x_1^3 + x_2^3 +  x_3^3$. So case \eqref{item plane} of Proposition \ref{proposition lines through p} really occurs.
\end{example}
\begin{lemma}\label{lemma h31}
Let $Y \subset \P^5$ be a cubic fourfold with isolated singularities and suppose that $Y$ does not contain a plane. If $p\in Y$ is a singular point, then $S:=M_1(Y,p)$ is a normal surface of degree $6$ in $\P^4$ and $S$ has only hypersurface singularities. If the singularities of $Y$ are simple, then the minimal resolution of $S$ is a K3 surface. Moreover, if $(Y,p)$ is an $A_1$-singularity and if $Y$ is smooth outside $p$, then $S$ is a smooth K3 surface.
\end{lemma}
\begin{proof}
For a suitable choice of coordinates $x_0,\ldots, x_5$ in $\P^5$ the cubic $Y$ is given by a polynomial of the form
\begin{equation}
\label{eq singular cubic}
f=x_5 f_2(x_0,\ldots, x_4) + f_3(x_0,\ldots, x_4)
\end{equation}
where $\deg f_i=i$ and $p=(0,0,0,0,0,1)$. 
 Then $S$ can be realized in the hyperplane $H \subset \P^5$ given by $x_5=0$ as the vanishing locus of the ideal $I_S=(f_2,f_3)$. As $Y$ does not contain a plane, it is neither reducible, nor a cone over a cubic threefold and thus $\dim S=2$ by Proposition~\ref{proposition lines through p}. So $S$ is a complete intersection and therefore $\deg S =6$. To show that $S$ is normal, it suffices to show that its singular locus has dimension $0$. From the Jacobi criterion we see that if $Y$ is singular at a point $p'=(p_1,p_2)$ different from $p$, then $p_1$ is a singular point of $S$. Conversely, if $S$ is singular at a point $p_1 \in H$, then either there is a singular point $p'\in Y$ with $\pi_p(p')=p_1$ where $\pi_p: Y\ohne p \to H$ is the projection from $p$ or the quadric $Q\subset H$ defined by $f_2=0$ is singular at $p_1$. Suppose that $\dim S^\sing = 1$. As the singularities of $Y$ are isolated, there has to be a curve $\Sigma \subset S$ such that $Q$ is singular along $\Sigma$. Hence, $Q$ is singular along $P_\Sigma = \erz \Sigma \subset H$ and $\rk Q \leq 4 - \dim P_\Sigma$. If $\Sigma$ is a line, then $Y$ contains the plane spanned by $\Sigma$ and $p$. If $\dim P_\Sigma \geq 2$, then $Q$ is reducible and $S$ contains a cubic surface. As cubic surfaces always contain lines, $Y$ would again contain a plane. So we see that $S$ has isolated singularities only. Let $p_1$ be a singular point of $S$. If $C=\{f_3=0\}$ and $Q$ were both singular at $p_1$, then $Y$ would be singular along the line $\ell_{pp_1}$. But $Y$ has isolated singularities, so we see that $S$ has embedding dimension $3$ at $p_1$. 
 
Suppose next that the singularities of $Y$ are simple. Then the same is true for $S$ by \cite[Theorem 2.1]{W99}.
Hence, its minimal resolution is a $K3$ surface. 
To prove the last statement, note that $(Y,p)$ is an $A_1$-singularity if and only if the Hessian of $f$ has rank $5$ which means that $f_2$ has to have rank $5$ as well. Thus, the Jacobian criterion shows that in this case $S$ is non-singular if and only if $Y$ is non-singular outside $p$.
\end{proof}

The Fano variety $M_1(Y)$ of lines on a smooth cubic fourfold $Y$ has been studied in \cite{BD}. It was shown that $M_1(Y)$ is an irreducible symplectic manifold if $Y$ is an arbitrary smooth cubic and that it is isomorphic to a Hilbert scheme of two points on a K3 surface of degree $14$ if $Y$ is Pfaffian. We will investigate the case of singular cubic fourfolds. This is needed for the proof of the main result, but should also be of independent interest.

As in section \ref{sec recap} we denote by $G$ the Grassmanian $\Gr(V,4)$ and by $V\to W$ its universal quotient bundle. Let $P \subset G\times \P^5$ be the universal family of lines on $\P^5=\P(V)$ and let $p_1:P \to G$ and $p_2:P\to \P^5$ be the projections. 
Let us  consider the exact sequence
\begin{equation}\label{eq fanos}
0 \to K \to S^3 V \tensor \sO_G  \to S^3 W \to 0.
\end{equation}
on $G$. The vector bundle $E:=S^3 W$ has rank four and a defining equation $f \in H^0(\P^5,\sO(3))=S^3 V$ for a cubic $Y\subset \P^5$ gives a section $s_f\in H^0(G,E)$ such that $F$ is the zero locus of $s_f$.
\begin{proposition}\label{proposition fano connected}
For every cubic hypersurface $Y \subset \P^5$ the Fano variety $M_1(Y) \subset G$ of lines on $Y$ is connected.
\end{proposition}
\begin{proof}
Let $K$ be as in \eqref{eq fanos}. Projection to the second factor makes the variety $\scrR := \P_G(K^*) \subset \P^5 \times \P(S^3V^\vee)$ into the universal family of Fano varieties of cubic fourfolds. The general fiber of $\scrR \to \P(S^3V^\vee)$, the Fano variety of lines on general smooth cubic, is connected. Let $\scrR \to \sQ \to \P(S^3V^\vee)$ be the Stein factorization. Then $\sQ \to \P(S^3V^\vee)$ is finite and birational, hence an isomorphism as $\P(S^3V^\vee)$ is a normal variety and $\sQ$ is connected. So every fiber of $\scrR \to \P(S^3V^\vee)$ is connected. This completes the proof.
\end{proof}

\begin{proposition}\label{proposition fano irreducible}
Let $Y \subset \P^5$ be a cubic hypersurface with only isolated singularities and suppose that $Y$ is not a cone over a cubic threefold. Then the Fano variety $M_1(Y) \subset \Gr(V,2)$ is a locally complete intersection, normal and irreducible and is singular at worst at the codimension $\leq 2$ set $\Sigma$ of lines passing through a singular point of $Y$.
\end{proposition}
\begin{proof}
As $M_1$ is connected by Proposition \ref{proposition fano connected} it suffices to show that it is irreducible along its singular locus. Let $E=S^3 W$ be the vector bundle of rank four from \eqref{eq fanos}. The defining equation $f \in H^0(\P^5,\sO(3))$ for $Y$ gives a section $s_f\in H^0(G,E)$ such that $F$ is the zero locus of $s_f$. Hence, every irreducible component of $M_1$ has dimension $\geq 4$. 
If $\ell \subset Y$ is a line which does not pass through any of the singular points of $Y$, then one shows as for smooth $Y$ that the normal bundle $N_{\ell/Y}$ is isomorphic to
\[
\sO(1)\oplus \sO^{\oplus 2} \textrm{ or } \sO(1)^{\oplus 2} \oplus \sO(-1).
\]
Consequently $H^1(N_{\ell/Y})=0$ and $M_1$ is smooth in $\ell$ by deformation theory. As $Y$ is not a cone the variety $M_1(Y,p)$ of lines through a point $p\in Y$ has dimension $\leq 2$. As all singular points are isolated, lines meeting $Y^\sing$ form a codimension $\geq 2$ subset and outside this set the variety $M_1(Y)$ is smooth. Therefore, all irreducible components are of pure dimension $4$ and $M_1$ is locally a complete intersection. By Serre's criterion the variety $M_1(Y)$ is normal and being connected it follows that it is irreducible.
\end{proof}
The following observation seems to be interesting in its own right. Recall from \cite{B} that a singular symplectic variety is a normal variety together with a symplectic form on its regular part which extends to a holomorphic $2$-form on every resolution of singularities.
 Recall that by Lemma \ref{lemma h31} the variety of lines on through a singular point is a singular K3 surface, i.e. its minimal resolution is K3.
\begin{theorem}\label{theorem m1 birational to hilb}
Let $Y \subset \P^5$ be a cubic hypersurface with only simple isolated singularities, let $p\in Y$ be a singular point and denote by $S:=M_1(Y,p)$. Then the Fano variety $M_1(Y) \subset \Gr(V,2)$ is a singular symplectic variety birational to $\Hilb^2(S)$.
\end{theorem}
\begin{proof}
By Proposition \ref{proposition fano irreducible} the variety $M_1$ is normal and irreducible. As before, the locus of lines meeting $Y^\sing$ will be denoted by $\Sigma \subset M_1$.
The construction of de Jong-Starr \cite{dJS} endows $M_1\ohne \Sigma$ with a symplectic $2$-form which extends by codimension reasons to a symplectic form $\sigma$ on $M_1(Y)^\reg$. 
Hence, we have $K_{M_1}=0$ and $\omega_{M_1}=\sO_{M_1}$ by \cite[Proposition 5.75]{KM}. Thus, $M_1$ is Gorenstein, as by Proposition \ref{proposition fano irreducible} it is a locally complete intersection.
We have to show that $\sigma$ extends to a regular $2$-form on some resolution of singularities. 
Admitting the birationality statement for the moment, we obtain the extension property as follows. Let us denote by $\tilde S$ the minimal resolution of $S$ which is a K3 surface by Lemma \ref{lemma h31}. Then there is a resolution of indeterminacy 
\[
\xymatrix{M_1&W\ar[l]_\pi\ar[r]&\Hilb^2(\tilde S)\\
}
\]
where $W$ is a smooth variety. As $\Hilb^2(\tilde S)$ is smooth and symplectic, we find $K_W \geq 0$. Then $0 \leq K_W=K_W - \pi^*K_{M_1}$ shows that $M_1$ has canonical singularities, as it is Gorenstein. So it has rational singularities and extension follows from \cite[Theorem 4]{Na01}.

It remains to show that $M_1$ is birational to $\Hilb^2(S)$. Let us take a cubic equation as \eqref{eq singular cubic} for $Y$. As before we realize $S$ as the zero locus of the ideal $(f_2,f_3)$ in the hyperplane  $x_5=0$. 
Let $K\subset Y$ denote the cone over $S$ with vertex $p$. Note that this is an ample Cartier divisor on $Y$ cut out by the equation $f_2=0$. Hence a generic line $\ell \subset Y$ intersects $K$ in exactly two points. This gives a rational map $M_1 \ratl \Hilb^2(S)$ via $\ell \mapsto \ell \cap K$. This map has a rational inverse given by $\xi \mapsto \ell_\xi$ where the intersection of $Y$ with the linear span $\langle \xi,p\rangle \isom \P^2$ is the cone over $\xi$ and a residual line $\ell_\xi$.
\end{proof}
\begin{remark}
The construction of the birational map is maybe not new but it turned out that giving a proof was easier than finding a reference. One of the important points in the proof of Theorem \ref{theorem m1 birational to hilb} was to show that $M_1$ has rational singularities. If we could show this directly, we would not need to resort to the explicit birational model $\Hilb^2(S)$ in order to prove extension of the two form. 
\end{remark}
\begin{remark}
The condition on not being a cone is certainly necessary as otherwise $M_1(Y)$ would contain a divisor $D$ isomorphic to a cubic threefold. If $M_1$ is normal, the restriction of a symplectic form on $M_1$ to $D$ would be non-zero at the generic point of $D$ for dimension reasons. This contradicts the fact that a cubic threefold does not contain any non-zero holomorphic $2$-form.
\end{remark}
Next we observe that most singular cubics are rational. This is well-known and not logically necessary for the rest of the paper. We include it nevertheless to underline the parallels to the Beauville-Donagi construction, see \cite[Proposition 5]{BD}. Moreover, it is interesting to note that current conjectures on when a cubic fourfold is rational can be phrased in terms of $M_1(Y)$, see \cite[Theorem 2]{A14}.\footnote{Observe however that in Addington's paper the divisor $\sC_6$ of singular cubics is not considered.}
\begin{lemma}\label{lemma rational}
Let $Y \subset \P^5$ be a cubic fourfold with an isolated singularity in $p \in Y$. If $Y$ is not a cone, then it is rational.
\end{lemma}
\begin{proof}
If $\ell \subset \P^5$ is a line through $p$ and not contained in $Y$ then $\mult_p \ell \leq \deg Y = 3$. Also $\mult_p \ell \geq  2$ as $p \in Y^\sing$. Projection from $p$ gives a morphism $\pi_p : Y \ohne \{p\} \to H \isom  \P^4$. As $Y$ is not a cone with vertex $p$, this map is generically quasi-finite. Hence, for a general line through $p$, $\mult_p \ell = 2$ and $\pi_p$ is birational.
\end{proof}
\begin{remark}
It is not coincidence that this argument breaks down for cones over smooth cubic threefolds. Indeed, rationality of the cone would imply that the corresponding cubic threefold is \emph{stably rational}. Smooth cubic threefolds are known to be irrational, but to the best of our knowledge not a single cubic threefold is proven to be stably irrational. Stable rationality would imply the existence of what is called a \emph{decomposition of the diagonal}, but this was shown to be a difficult problem in \cite[Theorems 1.4 and 1.5]{V}.
\end{remark}

\section{The locus of singular surfaces}\label{sec divisoren}
For the proof of normality of $M_3(Y)$ and of generic symplecticity of $Z'(Y)$ it is crucial to have dimension estimates on the locus of curves $C\in M_3$ such that the cubic surface $S=Y \cap P_C$ cut out by the linear span of $C$ has singularities worse than ADE. As we will suppose that $Y$ does not contain a plane, there are three classes of singular cubic surfaces: surfaces with ADE singularities only, simple elliptic surfaces, which are cones over a smooth plane cubic, and integral, non-normal surfaces, whose singular locus is necessarily a line.
We refer to \cite[\S 2]{LLSvS} for details and further references.
We will need the following lemmata.

\begin{lemma}\label{lemma cone stratum}
Suppose that the cubic hypersurface $Y\subset \P^5$ has only isolated singularities and does not contain any plane. Then the set of $P \in G$ such that $S_P:= Y\cap P$ is a cone over a plane curve is of dimension $\leq 4$.
\end{lemma}
\begin{proof}
First observe that $Y$ is not a cone over a cubic threefold, because the latter contains a line and thus $Y$ would contain the plane generated by such a line and the vertex of $Y$. Therefore, the variety $M_1(Y,p)$ of lines through a given point $p\in Y$ has dimension $\leq 2$ by Proposition \ref{proposition lines through p}.

By the argument of \cite[Proposition 4.3]{LLSvS} for a non-singular point $p\in Y$ there are at most two $P \subset \P^5$ such that $S_P$ is a cone with vertex $p$. A slight generalization of this yields our lemma. Let $p \in Y$ be a singular point. As before we choose a hyperplane $H \subset \P^5$ not containing $p$ and look at $M_1(Y,p) \subset X = Y\cap H$. We are looking for $\P^2 \isom P' \subset H$ such that $K:= P' \cap M_1(Y,p)$ is a curve of degree three. Let $P'$ with this property be given. By Proposition \ref{proposition lines through p} the variety $M_1(Y,p)=Q\cap C$ is the complete intersection of a quadric $Q$ and a cubic hypersurface $C$ in $\P^4$. As $P'\cap Q$ contains $K$ we must have $P'\subset Q$ so that $\rk Q \leq 4$. If $\rk Q \leq 2$, then Q contains a $\P^3$ and therefore $M_1(Y,p)$ contains a line $\ell$ and $Y$ contains the plane $\erz (p,\ell)$ contradicting the assumption. If $\rk Q$ is $3$ or $4$ then there are one respectively two one-dimensional families of planes in $Q$ so that the set of all $P$ for which $S_P$ is a cone with vertex $P$ is at most one dimensional. As the singularities of $Y$ are isolated, the claim follows.
\end{proof}

Non-normal cubic surfaces are exactly those, which are singular along a curve. In this case, their singular locus has to be a line.

\begin{lemma}\label{lemma non-normal stratum}
Suppose that the cubic hypersurface $Y\subset \P^5$ has only isolated singularities and does not contain any plane. Then the set of $P \in G$ such that the singular locus of $S_P=Y\cap P$ is of dimension one is of dimension $\leq 4$.
\end{lemma}
\begin{proof}
Suppose that the cubic surface $S_P$ is singular along the line $\ell = \P(L)$. If $\ell$ does not meet $Y^\sing$, then $P$ is uniquely determined by $\ell$ by the argument of \cite[Prop 4.2]{LLSvS}. Otherwise, write $\P^5=\P(V)$ and look at the exact sequence $0\to U \to V \to L\to 0$. As $\ell \subset Y$, the cubic polynomial $f$ defining $Y$ is in the kernel of $S^3V\to S^3W$. Denote by $f_1 \in U \tensor S^2L$ its leading term, which we interpret as a map $f_1:U^\vee \to S^2L$. Then $Y$ is singular along points of $\ell$ exactly at the common zeros of $\img f_1$ on $\ell$. Hence, $\rk f_1 \geq 1$. On the other side, let $V\to[\vphi] W$ be a $4$-dimensional quotient such that $P=\P(W)$ cuts $Y$ in a surface that is singular along $\ell$ and denote by $K$ the kernel of $\vphi$. Then $K\subset U$ and as $S_P$ is singular along $\ell$, the restriction of $f_1$ to $(U/K)^\vee$ vanishes. Therefore, $\rk f_1 \leq 2$. The biggest linear section of $Y$ which is singular in $\ell$ is cut out by $\P(V/N)$ where $N= (\img f_1)^\vee$. If $\rk f_1 = 2$, then $P=\P(V/N)$ is uniquely determined by $\ell$. If $\rk f_1 = 1$, then $P$ is a hyperplane in $\P(V/N)$ containing $\ell$ and the set of all these form $\P(U/N) \isom \P^2$. 
As $Y$ does not contain a plane we have $\dim M_1(Y,p)\leq 2$ for all $p\in Y$. The statement now follows as also those $P$ where $\rk f_1 = 1$ are parametrized by a variety of dimension $\leq 4$.
\end{proof}

The essential conclusion can be summed up in the following

\begin{corollary}\label{corollary not a divisor}
Suppose that the cubic hypersurface $Y\subset \P^5$ has only isolated singularities and does not contain any plane. Let $\Sigma \subset G$ be a subset of dimension $\leq 6$. Then the preimage of $\Sigma$ under $s:M_3(Y) \to G$ has dimension $\leq 8$.
\end{corollary}
\begin{proof}
Let $M_3 \to Z'$ be the $\P^2$-fibration from Proposition \ref{proposition p2fibration} and let $\delta:Z' \to G$ be the map to the Grassmanian. Let $\Sigma_{\textrm{cone}} \subset \Sigma$ be the subset of $P$ such that $Y \cap P$ is an elliptic cone singularity. Then $\dim \Sigma_{\textrm{cone}} \leq 4$ by Lemma \ref{lemma cone stratum}. By \cite[Corollary 3.11]{LLSvS}, $\delta^{-1}(\Sigma_{\textrm{cone}} ) \subset Z'$ has dimension $\leq 6$ and thus its preimage in $M_3$ has dimension $\leq 8$. The map $\delta$ mentioned there is exactly our $\delta_\univ$ from \eqref{eq universal p2fibration}, from which $\delta:Z' \to G$ is obtained as a pullback.
In the same vein one shows that the preimage of the non-normal locus has dimension $\leq 8$ thanks to Lemma \ref{lemma non-normal stratum}. As all other $P \in \Sigma$ cut $Y$ in a cubic surface with at worst ADE singularities, we conclude the proof by invoking once more \cite[Corollary 3.11]{LLSvS}.
\end{proof}

\section{The contraction and symplecticity}\label{sec kontraktion}

Now we are in good shape to analyze the space $M_3(Y)$ for singular cubic hypersurfaces.
\begin{lemma}\label{lemma reduced}
If $Y$ is a cubic fourfold with isolated singularities not containing a plane, then the space $M_3(Y)$ is connected and reduced and every irreducible component has dimension $10$.
\end{lemma}
\begin{proof}
By a Stein factorization argument similar to the one in Proposition \ref{proposition fano connected} we see that $M_3(Y)$ is connected if $M_3(Y')$ is connected for a generic cubic $Y'$. The latter holds true by \cite[Theorem A, p. 2]{LLSvS}. So connectedness of $M_3(Y)$ follows.

We observe that $M_3(Y)$ is cut out from the smooth variety $M_3(\P^5)$ by a section in a rank $10$ vector bundle. Thus, every irreducible component of $M_3(Y)$ has dimension $\geq 10$. It follows from Corollary \ref{corollary not a divisor} that every irreducible component dominates the Grassmanian, thus every irreducible component contains a curve $C \subset Y$ such that the cubic surface $S=P_C\cap Y$ is smooth. Hence, $C$ is an aCM curve. As the general member of the linear system of $C$ on $S$ is smooth, we may assume that $C$ is smooth. The proof of \cite[Theorem 4.7]{LLSvS}, step 3. to be precise, shows that $M_3$ is smooth in $C$ of dimension $10$. So every irreducible component has dimension $10$ and thus $M$ is locally a complete intersection. As every component contains a smooth point, $M_3$ is reduced. ($R_0$ and $S_1$)
\end{proof}

The identification of the non-CM locus also allows one to give a description of the singularities of $M_3(Y)$.
\begin{lemma}\label{lemma irreducible}
If $Y$ is a cubic fourfold with isolated singularities not containing a plane, then the space $M_3(Y)$ is normal and irreducible of dimension $10$.  In this case, the space $Z'(Y)$ is normal and irreducible of dimension $8$. 
\end{lemma}
\begin{proof}
We argue as in Proposition \ref{proposition fano irreducible}, that is, we give a bound for the dimension of the singular locus and deduce normality from Serre's criterion. $M_3(Y)$ is cut out from $M_3(\P^5)$ by a section in a rank $10$ vector bundle. Thus, every irreducible component has dimension $\geq 10$. We will show that the singular locus of $M_3(Y)$ has dimension $\leq 8$. 

First we will bound the dimension of the locus $\Sigma \subset M_3$ of curves meeting $Y^\sing$. The set of $P\in G$ containing a given point is isomorphic to a $\Gr(3,5)$ and thus of dimension $6$. As the singularities of $Y$ are isolated, we may conclude by Corollary \ref{corollary not a divisor} that $\Sigma$ has dimension $\leq 8$. 

Let $C$ be an aCM curve in $M_3\ohne \Sigma$. The proof of \cite[Theorem 4.7]{LLSvS} shows that $M_3$ is smooth in $C$ of dimension $10$. So the intersection of $M_3^\sing$ with the aCM-locus has dimension $\leq 8$. The non-CM locus is a Cartier divisor by Proposition \ref{proposition cartier divisor} and by Proposition \ref{proposition non cm} it may be identified with $J:=J(Y)$. As a variety is smooth at smooth points of a Cartier divisor, a singular point of $M_3$ which lies on $J$ is contained in the singular locus of $J$ and $\Sing J$ has dimension $7$.
Put together, $M_3$ is smooth in codimension $1$ and locally a complete intersection, therefore normal by Serre's criterion. In particular, $M_3$ is irreducible and has dimension $10$.
Clearly, the second statement is a consequence of the first, as smoothness descends along the $\P^2$-fibration $M_3 \to Z'$.
\end{proof}
We may draw the following important consequence.
\begin{lemma}\label{lemma canonical}
If $Y$ is a cubic fourfold with simple isolated singularities not containing a plane, then there is a holomorphic $2$-form $\sigma'$ on $Z'(Y)^\reg$ and a canonical divisor is given by $K_{Z'} = m D$ where $m \in \Z_{>0}$. In particular, $K_{Z'}$ is Cartier.
\end{lemma}
\begin{proof}
As in Theorem \ref{theorem m1 birational to hilb} the construction of de Jong and Starr \cite{dJS} yields a holomorphic $2$-form $\sigma_3$ on $M_3(Y)^\reg$. It descends along the $\P^2$-fibration $a:M_3 \to Z'$ to a holomorphic $2$-form $\sigma'$ on $(Z')^\reg$. 
If $\sigma'$ is degenerate, it is so along a divisor. Using an argument of de Jong and Starr \cite{dJS} as in \cite[4.4]{LLSvS} the form is non-degenerate at points $c$ where there is a point in $a^{-1}(c)$ corresponding to a smooth cubic $C \subset Y$ which does not pass through $Y^\sing$. This is always the case, if $c$ is of aCM type and the cubic surface $S_c=Y \cap P_c$ has ADE-singularities by \cite[Theorem 2.1]{LLSvS}. Here $P_c$ is the linear subspace corresponding to $\delta(c) \in G$. Surfaces $S_c$ with other singularities, elliptic cone or a non-normal integral surfaces, form closed subset of dimension $\leq 6$ by Lemmata \ref{lemma cone stratum} and \ref{lemma non-normal stratum} and \cite[Corollary 3.11]{LLSvS}. Thus, $\sigma'$ is degenerate only along the divisor $D\subset Z'$ of non-CM curves. That this is indeed the case follows as any resolution of singularities of $D\isom \Gr(\Omega_Y,3)$ has $H^{2,0}=0$. Finally, we already observed in  Proposition \ref{proposition cartier divisor} that $D$ is Cartier
\end{proof}
Recall from Proposition \ref{proposition cartier divisor} that there is an effective Cartier divisor $D \subset Z'(Y)$ parametrizing exactly the $\P^2$-families of non-CM curves. By Proposition \ref{proposition non cm} it is isomorphic to $\Gr(\Omega_Y,3)$ and we will identify the two once and for all.
Denote by $\pi: D \to Y$ the projection and consider the universal exact sequence
\begin{equation}\label{eq gromega}
0 \to U_\Omega \to \pi^*\Omega_Y \to Q_\Omega \to 0
\end{equation}
on $D$, where $Q_\Omega$ is locally free of rank $3$. 

We describe now a line bundle on $Z'(Y)$, whose linear system gives rise to a contraction and thereby produces a singular symplectic variety. Consider the map $\delta:Z'(Y) \to G$ ut $B:=\delta^*\det W$ and $L:=B(D)$.

\begin{lemma}\label{lemma morphism}
There is an exact sequence
\[
0 \to Q_\Omega \tensor \pi^*\sO_Y(1) \to \delta^*W\vert_D \to \pi^*\sO_Y(1) \to 0
\]
on $D$ and we have the following identities:
\[
\sO_{Z'}(D)\vert_D = \det U_\Omega^\vee, \quad B\vert_D = \pi^*\sO_Y(1) \tensor \det U_\Omega, \quad L\vert_D=\pi^*\sO_Y(1).
\]
\end{lemma}
Note that $\det U_\Omega^\vee$ is a line bundle, whereas $U_\Omega$ is a coherent sheaf of rank $1$.
\begin{proof}
The existence of the exact sequence and the first identity are shown as in \cite[Proposition 4.5]{LLSvS}. We immediately deduce the other identities using \eqref{eq gromega}.
\end{proof}
This enables us to perform the contraction.

\begin{lemma}\label{lemma contraction}
There is a divisorial contraction $b: Z'(Y) \to Z(Y)$ to a normal variety $Z(Y)$ and the exceptional divisor is $D$. The fibers of $\pi:D \to Y$ are contracted by $b$ and the image of $D$ under $b$ is isomorphic to $Y$ so that there is a closed immersion $\iota:Y \to Z(Y)$. Moreover, smooth points of $Y$ are smooth in $Z$.
\end{lemma}
\begin{proof}
The arguments of Lemmata 4.12 and 4.13 from \cite{LLSvS} work equally well in the case of a singular cubic $Y$, so that $L$ is nef and a curve $C \subset Z'$ with $\deg L\vert_C =0$ is necessarily contained in the fibers of $D\to Y$. 
It is evident from Corollary \ref{lemma morphism} that indeed for all curves $C$ contained in a fiber of $D\to Y$ we have $\deg L\vert_C =0$. We will show that there is $n\gg 0$ such that $L^{\tensor n}$ is basepoint-free and contracts $D$ along the morphism $D\to Y$. Outside $D$ the inclusion $B \to L$ is an isomorphism and as $B^{\tensor n}$ is basepoint-free on $Z'$ for some $n \gg 0$, the map 
composition $H^0(B^{\tensor n}) \to H^0(B^{\tensor n}) \to \C_z$ is surjective for every $z \in Z' \ohne D$.

So $L^{\tensor n}$ has no basepoints outside $Z$. We may assume that $n \geq m + 2$ where $m \in \Z_{> 0}$ is such that $K_{Z'}=m D$, see Lemma \ref{lemma canonical}. We consider the exact sequence
\[
0 \to L^{\tensor n}(-D) \to L^{\tensor n} \to L^{\tensor n}\vert_D \to 0.
\]
By the Kawamata-Viehweg vanishing theorem \cite[Theorem 2.70]{KM} we conclude that 
\[
H^1(L^{\tensor n}(-D)) = H^1(K_{Z'}\tensor L^{\tensor (n-m-1)}\tensor B^{\tensor (m+1)}) = 0
\]
and consequently $H^0(L^{\tensor n}) \to H^0(L^{\tensor n}\vert_D)$ is surjective and the restriction of $L^{\tensor n}$ to $D$ is $\pi^* \sO_Y(n)$ which is certainly basepoint-free. Thus, the big line bundle $L^{\tensor n}$ is basepoint-free for $n$ big enough and produces a birational morphism $b:Z'\to Z$ onto a normal variety such that a curve $C \subset Z'$ is contracted to a point if and only if it is contained in the fibers of the morphism $D\to Y$. As we already observed, $L^{\tensor n}\vert_D$ is the pullback of a very ample line bundle on $Y$ so that the image of $D$ under $b$ is identified with $Y$. The last statement is proven literally as in \cite[Proposition 4.17]{LLSvS}
\end{proof}

Our results sum up to the following

\begin{theorem}\label{theorem main}
If $Y$ is a cubic fourfold with simple isolated singularities and if $Y$ does not contain a plane, then there are morphisms
\[
M_3(Y) \to[a] Z'(Y) \to[b] Z(Y)
\]
such that $a$ is a smooth $\P^2$-fibration and $b$ is birational. Here $Z(Y)$ is a symplectic variety of dimension $8$. Moreover, there is a closed immersion $\iota: Y \into Z(Y)$ such that $Y$ is not contained in the singular locus of $Z(Y)$ and $Y$ is a Lagrangian subvariety. 
\end{theorem}
\begin{proof}
The morphism $a$ has been constructed in Proposition \ref{proposition p2fibration}, the contraction $b$ was obtained in Lemma \ref{lemma contraction}. For symplecticity observe that the form $\sigma'\in H^0((Z')^\reg,\Omega_{(Z')^\reg})$ from Lemma \ref{lemma canonical} is degenerate only along the irreducible divisor $D$. But $D$ is contracted to the $4$-dimensional variety $Y \subset Z$, hence the $\sigma'$ induces a symplectic form on $Z^\reg\ohne Y$ which extends to a symplectic form on $Z^\reg$ as $Y$ has codimension $4$ in $Z$. The only subtle point is the  extension of $\sigma$ as a holomorphic form to a resolution of singularities of $Z$. This is obtained exactly as in Theorem \ref{theorem m1 birational to hilb} where we use Starr's construction, see Theorem \ref{theorem starr}, to see that $Z$ has rational singularities. The embedding $\iota$ is Lagrangian, because $H^{2,0}(\wt Y)=0$ for any resolution of singularities $\wt Y \to Y$.
\end{proof}

\begin{corollary}\label{corollary namikawa}
If $Y$ is a cubic fourfold with simple isolated singularities and if $Y$ does not contain a plane, then the variety $Z(Y)$ admits a symplectic resolution, which is deformation equivalent to $Z(Y')$ for a smooth cubic fourfold $Y'$ not containing a plane.
\end{corollary}
\begin{proof}
The construction of $Z$ works in families. Therefore, $Z(Y)$ is smoothable by a flat deformation. Let $\wt Z \to Z$ be a $\Q$-factorial terminalization of $Z$ which exists by \cite[Corollary 1.4.3]{BCHM}. Clearly, $\wt Z$ is again a symplectic variety. Then $\wt Z$ is a crepant resolution of $Z$ by \cite[Corollary 2]{Na06}. 
The claim follows from \cite[Theorem 1, p. 109]{Na06}. 
\end{proof}
Obviously, the argument applies word by word for the variety of lines and yields
\begin{corollary}\label{corollary namikawa for m1}
If $Y$ is a cubic fourfold with isolated singularities, then its variety of lines $M_1(Y)$ admits a symplectic resolution, which is deformation equivalent to $M_1(Y')$ for a smooth cubic fourfold $Y'$.\qed
\end{corollary}

\section{Starr's fibration}\label{sec faserung}
We briefly recall Starr's unpublished construction of a rational Lagrangian fibration on $Z(Y)$ for cubics $Y$ with a single ordinary double point. I am grateful to Jason Starr for allowing me to do so. The fibration is constructed on a birational model of $M_3(Y)$ and then it is shown to descend to $Z'(Y)$, so birationally there is a fibration on $Z(Y)$. In particular, the argument shows that $Z(Y)$ is a deformation of $\Hilb^4(S)$ where $S$ is a K3 surface of degree $6$ in $\P^4$. 

Let $S \subset \P^4$ be a K3 surface of degree $6$ and consider a linear system $\abs B \isom (\P^4)^\vee$ of hyperplane sections. Let $\sB \to \abs B$ be the universal curve and let $\sM:=\Hilb^6(\sB/\abs B)$ be the Hilbert scheme of six points supported on curves of the linear system $\abs B$. We denote by $\sJ$ the compactification of the Jacobian fibration of $\sB \to \abs B$ in the moduli space of semi-stable sheaves on $S$. The variety $\sJ$ is known to be an irreducible symplectic manifold birational to $\Hilb^4(S)$ if for example all curves in the linear system are integral, see \cite{Bfas}. Hence, by Huybrechts' result \cite[Theorem 2.5]{Huy} the varieties $\sJ$ and $\Hilb^4(S)$ are deformation equivalent in this case. Note that we have a rational map
\[
\xymatrix{
\sM \ar@{-->}[rr] \ar[rd] && \sJ \ar[dl]\\
&\abs B & \\
}
\]
It is well defined over the open subset $U\subset \abs B$ of smooth curves, where it is a $\P^2$-fibration. Now Starr has shown
\begin{theorem}[Starr]\label{theorem starr}
Let $Y$ be a cubic fourfold with only simple isolated singularities and not containing a plane, let $p \in Y$ be a singular point and let $S := M_1(Y,p)$ be the variety of lines through $p$. Then there is a commutative diagram
\[
\xymatrix{
M_3(Y) \ar@{-->}[r] \ar[d] & \sM \ar[d]\\
Z(Y) \ar@{-->}[r] & \sJ \\
}
\]
where the horizontal maps are birational. In particular, there is a rational Lagrangian fibration $Z(Y) \ratl \P^4$.
\end{theorem}
\begin{proof}
Let $H \subset \P^5$ be a general hyperplane not containing $p$. Then $S$ can be realized in $X=Y\cap H$ and by Lemma \ref{lemma h31} its minimal resolution is a K3 surface so that the last statement of the theorem would follow from the existence of the commutative diagram. 

Let $C \in M_3(Y)$ be a smooth curve not contained in $H$ and such that $p\notin P_C = \erz C$. We will produce a pair $(B,D)$ where $B$ is a hyperplane section of $S$ and $D$ is an effective divisor of degree $6$ on $B$.
Let $\pi_p:Y\ohne p \to H$ be the projection from $p$ and consider the twisted cubic $C':=\pi_p(C)\subset H$. We denote by $P\subset H$ its span. Then $B$ is obtained as the intersection $P\cap S$. As to $D$, note that in general $C'$ is not contained in $Y$. Therefore, $X$ intersects $C'$ in $9$ points, three of which lie on $C$.  Let $D$ be the six points in $X \cap C'$ not lying on $C$. As $D \subset P$, the inclusion $D\subset B$ follows if the cone over $D$ with vertex $p$, which is a union of six lines, is contained in $Y$. But as these lines contain a point of $C'$, a point of $C$ not in $H$ and the point $p$ (necessarily with multiplicity $\geq 2$), they have to be contained in $Y$. Thus, we obtain a rational map $M_3 \ratl \sM$ by $C \mapsto (B,D)$.

Conversely, given a generic hyperplane section $B$ of $S$ and a generic effective divisor $D$ of degree $6$ on $B$ we look at the span $P$ of $B$ in $H$, which is a $\P^3$. There is a unique twisted cubic $C'$ through $D$. Moreover, $C'\cap B$ is precisely $D$ as $B$ lies on a unique quadric surface in $P$ which intersects $C'$ in no more than $6$ points.
The curve $C'$ is not the one we are looking for as it is not contained in $Y$. Let $K$ be the cone over $C'$ with vertex $p$, this is a cubic surface in $\P^4 = \erz\langle B,p\rangle$. Its intersection with $Y$ is a curve of degree $9$ and contains the cone over $D$ which consists of six lines coincident at $p$. Thus, the residual curve $C$ is of degree $3$ and we will show that it has to be a twisted cubic curve. Indeed, if $q \in C'$ is not in $B$, then the line $\ell_q$ connecting $p$ with $q$ is not contained in $Y$. As it has multiplicity $2$ in $p$, it intersects $Y$ in exactly one point. Thus, the residual curve $C$ maps bijectively to $C'$ under the projection from $p$. So we obtain a rational map $\sM \ratl M_3$, $(B,D) \mapsto C$.
Clearly, these constructions are inverse to each other. As the maps $\sM \ratl \sJ$ and $M_3 \to Z'$ are $\P^2$-fibrations and as a $\P^2$ cannot map non-trivially to an abelian variety, this birational map descends to a birational map $Z' \ratl \sJ$.
\end{proof}

\begin{remark}
Starr's theorem gives another proof of irreducibility of $Z(Y)$ and $M_3(Y)$ if $Y$ has simple singularities. Indeed, every irreducible component of $M_3(Y)$ has dimension $\geq 10$ as is observed in the proof of Lemma \ref{lemma reduced} so that the map $M_3 \ratl  \sM$ is defined on a dense open subset of $M_3$. As $\sM$ is irreducible and birational to $M_3$, the latter is also irreducible.
\end{remark}

This allows us to draw the following conclusion for $Z(Y)$.

\begin{corollary}\label{corollary starr}
If $Y'$ is a smooth cubic fourfold not containing a plane, then $Z_3(Y')$ is deformation equivalent to the Hilbert scheme of $4$ points on a K3 surface.
\end{corollary}
\begin{proof}
Take a generic cubic fourfold $Y$ with an $A_1$-singularity. Then by Corollary \ref{corollary namikawa}, the variety $Z(Y')$ is deformation equivalent with a crepant resolution $\wt{Z(Y)}$ of $Z(Y)$. The variety $\wt{Z(Y)}$ is birational to the Jacobian fibration $\sJ \to \P^4$ by Starr's result, Theorem \ref{theorem starr}, and $\sJ$ is known to be a deformation of the Hilbert scheme of four points on a K3 surface, see \cite{Bfas}. Summing up, we have
\[
Z(Y') \sim_{defo} \wt{Z(Y)} \stackrel{birat}{\ratl} \sJ \sim_{defo} \Hilb^4(K3).
\]
As $\wt{Z(Y)}$ and $\sJ$ are both irreducible symplectic manifolds, the claim now follows from Huybrechts' theorem \cite[Theorem 2.5]{Huy} saying that two birational (smooth) irreducible symplectic manifolds are deformation equivalent.
\end{proof}


\end{document}